\documentclass[12pt,a4paper]{amsart}
\usepackage{graphicx}
\usepackage{amsmath}
\usepackage{amsthm}
\usepackage{amssymb}
\usepackage{amsfonts}
\usepackage{latexsym}
\usepackage{color}

\newcommand{\Ds}{\displaystyle}

\newcommand{\Norm}[1]{{\left\|{#1}\right\|}}

\newcommand{\PP}{{\mathcal P}}
\newcommand{\Pn}{{\mathcal P}_n}
\newcommand{\PK}{{\mathcal P}_n(K)}
\newcommand{\ze}{\zeta}
\newcommand{\DK}{\partial K}
\newcommand{\HH}{{\mathcal G}}

\newcommand{\HP}{{\mathbb H}}

\newcommand{\Z}{{\mathcal Z}}

\newcommand{\W}{{\mathcal W}}
\newcommand{\RR}{{\mathbb R}}
\newcommand{\CC}{{\mathbb C}}

\newcommand{\NN}{{\mathbb N}}
\newcommand{\TT}{{\mathbb T}}
\newcommand{\DD}{{\mathbb D}}
\newcommand{\II}{{\mathbb I}}

\newcommand{\diam}{{\rm diam\,}}
\newcommand{\width}{{\rm width\,}}

\newtheorem{theorem}{Theorem}

\newtheorem{lemma}{Lemma}

\newtheorem{conjecture}{Conjecture}
\theoremstyle{definition}

\newtheorem{remark}{Remark}

\newcommand{\de}{\delta}

\newcommand{\ff}{\varphi}

\newcounter{othm}
\def\theothm{\Alph{othm}}
\newenvironment{othm}{
  \em
  \vskip 0.10in
  \refstepcounter{othm}
  \noindent{\bf Theorem\ \theothm .}
}{\vskip 0.10in}

\newcounter{rev}
\newcounter{rep}
\definecolor{darkgreen}{rgb}{0,0.6,0}

 \let\le\leqslant \let\leq\le
 \let\ge\geqslant \let\geq\ge

\reversemarginpar 

\begin{document}

\title[$L^q$ Tur\'an inequalities on convex polygons]{Tur\'an type oscillation inequalities in $L^q$ norm\\ on the boundary of convex polygonal domains}

\author{Polina Glazyrina, Szil\'ard Gy. R\'ev\'esz}

\address
{Polina Yu. Glazyrina \newline  \indent Institute of Natural Sciences and Mathematics, \newline  \indent   Ekaterinburg, Mira str. 19,\newline  \indent 620002 RUSSIA \newline  \indent  and \newline  \indent
Institute of Mathematics and Mechanics, \newline  \indent Ural Branch of the Russian Academy of Sciences, \newline  \indent
Ekaterinburg,   S.\,Kovalevskaya str. 16,
\newline  \indent 620077 RUSSIA}
\email{polina.glazyrina@urfu.ru}

\address{Szil\'ard Gy. R\'ev\'esz
\newline  \indent HUN-REN A. R\'enyi Institute of Mathematics \newline \indent Budapest, Re\'altanoda utca 13-15. \newline \indent 1053 HUNGARY} \email{revesz.szilard@renyi.hu}

\date{\today}

\begin{abstract} In 1939 P\'al Tur\'an and J\'anos Er\H{o}d initiated the study of lower estimations of maximum norms of derivatives of polynomials, in terms of the maximum norms of the polynomials themselves, on convex domains of the complex plane. As a matter of normalization they considered the family $\PK$ of degree $n$ polynomials with all zeros lying in the given convex, compact subset $K\Subset \CC$. While Tur\'an obtained the first results for the interval $\II:=[-1,1]$ and the disk $\DD:=\{ z\in \CC~:~ |z|\le 1\}$, Er\H{o}d extended investigations to other compact convex domains, too. The order of the optimal constant was found to be $\sqrt{n}$ for $\II$ and $n$ for $\DD$.
It took until 2006 to clarify that all compact convex \emph{domains} (with nonempty interior), follow the pattern of the disk, and admit an order $n$ inequality.

For $L^q(\partial K)$ norms with any $1\le  q <\infty$ we obtained order $n$ results for various classes of domains. Further, in the generality of all convex, compact domains we could show a $c n/\log n$ lower bound together with an $O(n)$ upper bound for the optimal constant. Also, we conjectured that all compact convex domains admit an order $n$ Tur\'an type inequality. Here we prove this for all \emph{polygonal} convex domains and any $0<  q <\infty$.
\end{abstract}

\maketitle

\let\oldfootnote\thefootnote
\def\thefootnote{}
\footnotetext{P. Yu. Glazyrina was supported by the Ministry of Science and Higher Education of the Russian Federation (Ural Federal University Program of Development within the Priority-2030 Program).

Sz. Gy.~R\'{e}v\'{e}sz was supported by the Hungarian National Research Innovation and Development Office, projects No. K-146387, K-147153 and E-151341.
}
\let\thefootnote\oldfootnote

\bigskip
\bigskip



\tableofcontents


\section{Introduction}

\subsection{The oscillation of a polynomial in maximum norm}

At the turn of the 19th and 20th centuries, the first estimates of the derivative of a polynomial via the maximum of its
values  appeared.
They were obtain by A. Markov in 1889, for  algebraic polynomials on an interval,
and by Bernstein and M.~Riesz in 1914, for  trigonometric polynomials
on $[0,2\pi]$ and algebraic polynomials on the unit circle.
In 1923,  Szeg\H o  \cite{Szego23} obtained an estimate for a large class of (not necessarily convex, but piecewise smooth)
 domains. Namely, if $K\subset \CC$ is a piecewise smooth simply connected domain, with its boundary consisting of finitely many analytic
Jordan arcs, and if the maximum of the outer angles at the joining
 vertices of these arcs is\footnote{If the domain is bounded, then for all directions it has supporting lines,
whence there are points where the outer angle is at least $\pi$.} $\beta \in [\pi,2\pi]$, then the domain admits
a Markov type inequality of the form $\Norm{p'}_K\le c_K n^{\beta/\pi} \Norm{p}_K$
for any polynomial $p$ of degree $n$.
Here the norm $\Norm{\cdot}:=\Norm{\cdot}_K$ denotes sup norm over values attained on $K$.
 This inequality is essentially sharp for all such domains. In particular, this immediately implies that for \emph{analytically smooth}
convex domains the  Markov factor is $O(n)$.
For the unit disk
$
\DD:=\{z\in \CC~:~ |z|\le 1\}
$
even the exact inequality is well-known:
\begin{equation*}
\Norm{p'}_\DD\le n \Norm{p}_\DD.
\end{equation*}
This was conjectured, and almost proved, by Bernstein \cite{Bernstein, Bernstein_com}; for the first published proof see \cite{Riesz}.
Similarly, the precise result is also classical for the unit interval
$
\II:=[-1,1]
$
-- then we have Markov's inequality $\Norm{p'}_\II\le n^2 \Norm{p}_\II$, which is sharp\footnote{Note that in this case the outer angles at the break-points of the piecewise smooth boundary are exactly $2\pi$ at each end.}, see
\cite{Markov}.

In 1939 P\'al Tur\'an started to study converse inequalities of
the form $$\Norm{p'}_K\ge c_K n^A \Norm{p}_K.$$ Clearly such a
converse can only hold if further restrictions are imposed on
the occurring polynomials $p$. Tur\'an assumed that all zeroes
of the polynomials belong to $K$. So denote the set of complex
(algebraic) polynomials of degree (exactly) $n$ as $\PP_n$, and
the subset with all the $n$ (complex) roots in some set
$K\subset\CC$ by $\PK$.

\begin{othm}{\bf(Tur\'an, \cite[p. 90]{Tur}).}\label{oth:Turandisk} If $p\in \PP_n(\DD)$,  then we have
\begin{equation}\label{Turandisk}
\Norm{p'}_\DD\ge \frac n2 \Norm{p}_\DD~.
\end{equation}
\end{othm}

\begin{othm}{\bf(Tur\'an, \cite[p. 91]{Tur}).}\label{oth:Turanint} If $p\in\PP_n(\II)$, then we have
\begin{equation*}
\Norm{p'}_\II\ge \frac {\sqrt{n}}{6} \Norm{p}_\II~.
\end{equation*}
\end{othm}

Inequality \eqref{Turandisk} of Theorem \ref{oth:Turandisk} is best possible.
Regarding Theorem \ref{oth:Turanint}, Tur\'an pointed out by
example of $(1-x^2)^{n}$ that the $\sqrt{n}$ order is sharp.
Some slightly improved constants can be found in
\cite{BabenkoMN86} and \cite{LP}, however, the exact value of
the constants and the corresponding extremal polynomials were
already computed for all fixed $n$ by Er\H{o}d in \cite{Er}.

Now we are going to describe  Tur\'an-type
inequalities  for general convex sets. Denote by $\Gamma:=\partial K$ the boundary of
$K$.
The (normalized) quantity under our study
is the ``inverse Markov factor" or ``oscillation factor"
\begin{equation}\label{Mdef}
M_{n,q}(K):=\inf_{p\in \PK} M_q(p) \qquad \text{\rm with} \qquad
M_q(p):=\frac{\Norm{p'}_{L^q(\Gamma)}}{\Norm{p}_{L^q(\Gamma)}},
\end{equation}
where, as usual,
\begin{align*}
\Norm{p}_{q}:&=\Norm{p}_{L^q(\Gamma)}:=\left(\int_{\Gamma} |p(z)|^q|dz|\right)^{1/q} \quad (0<q<\infty), \notag
\\ \Norm{p}_{\infty}:&=\Norm{p}_{L^\infty(\Gamma)}:=\Norm{p}_K=\sup_{z\in \Gamma}|p(z)|=\sup_{z\in K}|p(z)|.
\end{align*}

Drawing from the work of Tur\'an, Er\H od \cite[p. 74]{Er} already addressed the question: ``For what kind of domains does the method of Tur\'an apply?'' Clearly, by ``applies'' he meant that it provides order $n$ oscillation for the derivative. Moreover, he introduced new ideas into the investigation -- including the application of Chebyshev's Inequality \eqref{transfdiam} below -- so clearly he did not simply pursue the effect of Tur\'an's original methods, but was indeed after the right oscillation order of general domains. In particular, he showed
on p. 77 of \cite{Er} the following.

\begin{othm}{\bf(Er\H od).}\label{oth:transfquarter} Let $K$ be any convex domain bounded by finitely many Jordan arcs, joining at vertices with angles $<\pi$, with all the arcs being $C^2$-smooth and being either straight lines of length $<\Delta(K)$, where $\Delta(K)$ stands for the transfinite diameter of $K$, or having positive curvature bounded away from $0$ by a fixed constant $\kappa>0$. Then there is a constant $c(K)>0$, such that $M_{n,\infty}(K)\geq c(K) n$ for all $n\in\NN$.
\end{othm}

Note that this latter result of Er\H od incorporates regular $k$-gons $G_k$ for large enough $k$, but not the square $Q=G_4$, because the side length $h$ of a square is larger than the quarter of the transfinite diameter $\Delta$:
actually, with $\Gamma$ denoting Euler's Gamma function and $h$ standing for the side length of $G_k$, we have
$$
\Delta(G_k)= \frac{\Gamma(1/k)}{\sqrt{\pi}2^{1+2/k}\Gamma(1/2+1/k)} h,
$$
(see e.g. \cite[p. 135]{Rans}), so $\Delta(G_k) > h$ iff $k\geq 7$. This implies
$M_n(G_k) \ge c_k n$ for $k\ge 7$.

In \cite{E}, Erd\'elyi proved order $n$ oscillation for the square\footnote{Erd\'elyi also proves similar results on rhombuses, under the further condition of some symmetry of the polynomials in consideration -- e.g. if the polynomials are real, or odd. Note also that his work on the topic preceded \cite{Rev2} and apparently was accomplished without being aware of details of \cite{Er}.} $Q=G_4$, too. A result of \cite{Rev1} also implied $M_n(G_k) \ge c_k n$ for $k\ge 4$, but still not for a triangle.

In the full generality of all compact convex sets, however, only an order $\sqrt{n}$ Tur\'an-Er\H od type inequality was shown prior to 2006, see \cite[Theorem 3.2]{LP}. Note that this is sharp for $\II$, but later it turned out that for all other compact convex domains the true order of magnitude is similar to the disk case.

Clearly, assuming boundedness is natural, since all polynomials have
$\Norm{p_n}_K=\infty$ when the set $K$ is unbounded. Also, restricting ourselves to \emph{closed} bounded sets -- i.e., to compact sets -- does not change the $\sup $ norm of polynomials under study, as all polynomials are continuous.

Recall that the term {\em convex domain} stands for a compact, convex subset of $\CC$ having nonempty interior. That is, assuming that $K$ is a compact convex \emph{domain}, not just a compact convex \emph{set}, means that we exclude only the case of the interval, for which already Tur\'an clarified that the order of oscillation is precisely $\sqrt{n}$.

So in order to clarify the order of oscillation for all compact convex sets it remains to clarify the order of oscillation for compact convex domains. A solution was published in 2006, see \cite{Rev2}.

To study \eqref{Mdef} some geometric parameters of the convex domain $K$
are involved naturally. We write $d:=d(K):=\diam (K)$ for the
{\em diameter} of $K$, and $w:=w(K):=\width(K)$ for the {\em
minimal width} of $K$. That is,
\begin{equation*}
d(K):= \max_{z', z''\in K} |z'-z''|,
\end{equation*}
\begin{equation*}
w(K):= \min_{\gamma\in [-\pi,\pi]} \left( \max_{z\in K} \Re
(ze^{i\gamma}) - \min_{z\in K} \Re (ze^{i\gamma}) \right).
\end{equation*}
Since $K$ has nonempty interior, we have $0<w(K)\le d(K) <\infty$.

\begin{othm}{\bf(Hal\'{a}sz--R\'ev\'esz).} 
 Let $K\subset \CC$ be any bounded convex domain.
Then for all  $p\in
\PK$ we have
\begin{equation*}
\Norm{p'}_K\ge C(K) n  \Norm{p}_K~\qquad \text{\rm with} \qquad C(K)= 0.0003 \frac{w(K)}{d^2(K)}.
\end{equation*}
\end{othm}
\begin{remark}
This indeed provides the precise order, for an even larger  than $n$ order cannot occur, not for any particular compact set. Namely, let $K\subset\CC$ be any compact set with diameter $d:=\diam(K)$. Then for all $n$ there exists a polynomial $p\in\PK$ of degree exactly $n$ satisfying
\begin{equation*}
\Norm{p'}_{K} \leq ~  \frac{1}{\diam(K)} n~ {\Norm{p}}_{K} .
\end{equation*}
Indeed, considering a diameter $[z_0,w_0]$ and the polynomial $p(z)=(z-z_0)^d$, the respective norm is $\|p\|_{K}=d^n$ while the derivative norm becomes $\|p'\|_{K}=nd^{n-1}$, both attained at $w_0\in K$.
\end{remark}
So, this settles the question of the \emph{order}, but not the precise \emph{dependence on the geometry}. However, up to an \emph{absolute constant factor}, even the dependence on the geometrical features of the domain was also clarified in \cite{Rev2}.
\begin{othm}{\bf(R\'ev\'esz).} 
Let $K\subset\CC$ be any
compact, connected set with diameter $d$ and minimal width $w$.
Then for all $n>n_0:=n_0(K):= 2 (d/16w)^2 \log (d/16w)$ there
exists a polynomial $p\in\PK$  satisfying
\begin{equation*}
\Norm{p'}_K \leq  C'(K) n \Norm{p}_K \qquad \text{\rm with}
\qquad C'(K):= 600 ~\frac{w(K)}{d^2(K)}~.
\end{equation*}
\end{othm}

\begin{remark}
Komarov \cite{Komarov2026} recently obtained estimates showing the relationship between $n$ and the geometric characteristics of $K$ as the width of $K$ tends to $0$:
$$
28\left(\dfrac{w(K)}{d^2(2)}n+\dfrac{1}{d(K)}\sqrt n\right)\ge  M_{n,\infty}(K)\ge 0.00015\left(\dfrac{w(K)}{d^2(2)}n+\dfrac{1}{d(K)}\sqrt n\right).
$$
These estimates are valid both for the case of a domain, i.e. $w>0$, and an interval, i.e. $w=0$, as well as for all  $n$ and $d>0$.
Also note that for smooth domains there is another description of dependence on geometry, which refers to curvature \cite{Rev3}.
\end{remark}

Actually a sharper pointwise inequality holds \emph{at all points} of
$\partial \DD$. Namely, for $p\in \PK$ we have
\begin{equation*}
|p'(z)| \ge \frac n{2} |p(z)|, \quad |z|=1,
\end{equation*}
and as a corollary, for any $q>0$,
\begin{equation*}
\left(\int_{|z|=1}|p'(z)|^q|dz|\right)^{1/q} \ge  \frac n{2}\left( \int_{|z|=1}|p(z)|^q|dz|\right)^{1/q}.
\end{equation*}
In other words, Theorem \ref{oth:Turandisk}
extends to all integral norms on the perimeter, including all $L^q(\partial \DD)$-norms, and we have for all polynomials $p\in\PK$
\begin{equation*}
\Norm{p'}_{L^q(\partial
\DD)}\ge \frac n2 \Norm{p}_{L^q(\partial \DD)}, \qquad
M_{n,q}(\DD) \geq \frac{n}{2}~.
\end{equation*}
The same way, for so-called $R$-circular domains\footnote{This term was introduced by Levenberg and Poletsky \cite{LP} and it means that for any boundary point $z \in \partial K$ there exists a disk $D_R$ of radius $R$ with $z \in \partial D_R$ and $K\subset D_R$.} the result of Theorem \ref{oth:Turandisk}  extends as
\begin{equation*}
\Norm{p'}_{L^q(\partial K)}\ge \frac{n}{2R} \Norm{p}_{L^q(\partial K)}, \qquad M_{n,q}(K) \geq \frac{n}{2R}~.
\end{equation*}
In case we discuss maximum norms, one can assume that $|p(z)|$ is maximal, and it suffices to obtain a lower estimation of $|p'(z)|$ only at such a special point -- for general norms, however, this is not sufficient. The above results work only for we have a pointwise inequality of the same strength \emph{everywhere}, or almost everywhere. The situation becomes considerably more difficult, when such a statement cannot be proved. E.g. if the domain in question is not strictly convex, i.e. if there is a line segment on the boundary, then the zeroes of the polynomial can be arranged so that even some zeroes of the derivative lie on the boundary, and at such points $p'(z)$ -- even $p'(z)/p(z)$ -- can vanish. As a result, at such points no fixed lower estimation can be guaranteed, and lacking a uniformly valid pointwise comparision of $p'$ and $p$, a direct conclusion cannot be drawn either.

This explains why the case of the interval $\II$ already proved to be much more complicated for the integral mean norms.

In a series of papers \cite{Zhou84, Zhou86, Zhou92, Zhou93, Zhou95}, Zhou proved the inequality
\begin{equation*}
\left( \int_{-1}^1 |p^{(k)}(x)|^p dx \right)^{1/p}\ge C^{(k)}_{p,q}(n)
\left( \int_{-1}^1 |p(x)|^q dx \right)^{1/q},
\end{equation*}
with $k=1$, $C^{(1)}_{p,q}(n)=c_{p,q} \left(\sqrt{n}\right)^{1-1/p+1/q}$ and
$0<p\le q\le \infty,$ ${1-1/p+1/q\ge 0}$.

The best possible constants $C^{(k)}_{p,q}(n)$ were found by Babenko and Pichugov~\cite{BabenkoU86} for $p=q=\infty,$ $k=2$, by Bojanov \cite{Bojanov93} for $1\le p \le \infty,$ $q=\infty$, $1\le k\le n$, and by Varma~\cite{Varma88_83} in the case of $p=q=2,$ $k=1$.

Exact Tur\'an-type inequalities for trigonometric polynomials in different $L^q$-metrics on $\TT$ were proved in \cite{BabenkoU86, BabenkoMN86, Bojanov96, Tyrygin88, TyryginDAN88, KBL}.

Other inequalities on $\II$, $\DD$, the positive semiaxes, or on $\TT$ in various weighted $L^q$-metrics can be found in \cite{Varma79, UV96, WangZhou02, KBL}.

Note that investigations have been started in the direction of reverse Markov--Bernstein inequalities under modified conditions like replacing $\Pn$  with spaces of incomplete polynomials \cite{Erdelyi2020}, or postulating the restriction that the zeros belong to the prescribed set $K$ with allowing a few exceptions \cite{Komarov2019, Erdelyi2021}.

As said above, we also have a direct result for $R$-circular domains, and $R$-circularity could be ascertained by some conditions on the curvature. However, apart from these, for general domains, the situation was much less clear. A fully general estimate, with almost optimal order, was published as Theorem 1' in \cite{GR-3}.

\begin{othm}{\bf(Glazyrina--R\'ev\'esz).}\label{th:nlogn} Let $1\le q<\infty$ and let $K\subset\CC$ be any
compact, convex domain with diameter $d$ and minimal width $w$.
Then for all $n>n_0:=n_0(K):= \max\left(10^{21}, (d/w)^5 \right)$ and for all polynomial $p\in\PK$ we have
\begin{equation*}
\Norm{p'}_{L^q(\partial K)} \geq  \frac{1}{240 ~000} ~\frac{w^2}{d^3} ~\frac{n}{\log n} \Norm{p}_{L^q(\partial K)}.
\end{equation*}
\end{othm}

This is close to optimal, as was shown in \cite{GGR}. Here we only quote a corollary of the more general result there, see Corollary 1 of \cite{GGR}. An earlier version with some less explicit and worse constant was given as Theorem 5 in \cite{GR-2}.

\begin{othm}{\bf(Glazyrina--Goryacheva--R\'ev\'esz).}\label{th:sharpq} Let $K\subset \CC$ be a compact convex subset of $\CC$ having width $w>0$ and diameter $d$.
Let $0<q\le \infty$, and $$\displaystyle{ n \ge 2(1+1/q)\frac{d^2}{w^2} \log \frac{d}{w}}.$$
If $\mu$ is  the linear Lebesgue measure on the boundary of $K$ (arc length measure $\ell$), then for all $p \in \Pn(K)$ we have
$$
\Norm{p'}_{L^q(\partial K)} \le C_q \frac{w}{d^2} n \Norm{p}_{L^q(\partial K)},
$$
where
\begin{equation*}
 C_q:= 121 \frac{3q+2+2\sqrt{q^2+3q+1}}{5q} \left(3+2q+2\sqrt{q^2+3q+1}\right)^{1/q}.
\end{equation*}
\end{othm}

Also we conjectured that the right order of Tur\'an-Er\H od type oscillation is always $n$, for all convex compact domains and for all exponents $q>0$. This was based, on the already mentioned general results, and several partial results, for various classes of compact convex domains, where we could indeed prove $\Norm{p'}_{L^q(\partial K)} \gg n \Norm{p}_{L^q(\partial K)}$. These include generalized Er\H od type domains including the class in Theorem \ref{oth:transfquarter}, see \cite{GR-1}, and some others, too.

We recall here only one particular result, which will  be needed  in the forthcoming proof as well.
Let us introduce  the notion of the  {\it local depth}.  The local depth $h(\zeta,K)$ is  the maximum of the length(s) of intersections of $K$ and normal line(s) at $\zeta$, that is, maximum of chord lengths emanating from $\zeta$ and perpendicular to some supporting lines.
We  also  need the set $\HH(p)$, which we define as
\begin{equation}\label{eq:Hsetdef}
\HH:=\HH(p):=\HH_K^q(p):=\left\{\zeta\in \Gamma~: ~ |p(\zeta)| > \frac{\lambda}{n^{2/q}} \|p\|_\infty\right\},
\quad \lambda:=\frac{1}{\left(8\pi(q+1) \right)^{1/q}}.
\end{equation}

\begin{othm}\label{th:localdepth} Let $n\in \NN$ and $p\in\PK.$
Then for any $n\ge n_0(K)$ and $\zeta\in\HH(p)$ it holds
\begin{equation*}
|p'(\ze)| \ge  \frac{h^4(\zeta,K)}{1500 d^5(K)}~n~|p(\ze)|.
\end{equation*}
\end{othm}


This was important as a key tool to obtain a result for domains with \emph{fixed positive depth}, that is, with $h_K:=\inf_{\ze\in \DK} h(\ze,K)>0$. When this happens to hold, the result  immediately implies that\footnote{To be precise, this was proved in \cite{GR-2} only for $q\ge 1$, unfortunately. However, the assumption that $q \ge 1$ was used only in settling the case of small $n$; for $n \le n_0:=32 d^4/h^4$ Lemma 3 of \cite{GR-2} was used to get the statement. Therefore, one can extend the result to $0<q<1$, too, at the expense of assuming $n\ge n_0$ (with the above value of $n_0$, depending only on $\zeta$ and $K$ through $h$ but not on $q$).} $\Norm{p'}_{L^q(\partial K)} \gg n \Norm{p}_{L^q(\partial K)}$. Note that for convex polygonal domains $h_K>0$ is equivalent to assuming that the polygon has no acute angles, see Proposition 1 of \cite{GR-2}. In particular, among regular $k$-gons already for $k\ge 4$ we obtain the desired order $n$ oscillation result. In view of this, regarding polygonal domains there remained only polygons with acute angles to handle. This handling of acute angles is the crux of our present work. The solution is based on one of the key ideas of \cite{Rev2}, the so-called "tilted normal estimate lemma", but essential changes are necessary for making this idea work in our situation, too.

\begin{theorem}\label{th:polygon} Let    $K\Subset \CC$ be a convex polygonal domain. Then for any $0<q<\infty$ there exists a positive constant $c_K=c_K(q)$ and $n_0:=n_0(K)$ such that for all $n\ge n_0$ and $p\in\PK $
\begin{equation*}
\|p'\|_{L^q(\partial K)} \ge c_K n
\|p\|_{L^q(\partial K)}.
\end{equation*}
\end{theorem}

\section{Technical preparations for the investigation of $L^q(\partial K)$ norms}

We will use an idea going back to Er\H od: the use of Chebyshev's classical theorem in the general form found by Fekete. This is formulated with the use of the \emph{transfinite diameter} $\Delta(E)$ of a compact set $E$.

\begin{lemma}{\bf (Faber, Fekete, Szeg\H o).}\label{l:transfinitediam}
Let $M \Subset \CC$ be any compact set. Then for all
$k\in\NN$ we have
\begin{equation}\label{transfdiam}
\min_{w_1,\dots,w_k\in \CC} \max_{z\in M} \left| \prod_{j=1}^k
(z-w_j) \right| \ge \Delta(M)^k.
\end{equation}
\end{lemma}

\begin{proof} Regarding the formulation in Lemma \ref{l:transfinitediam} cf.
Theorem 5.5.4. (a) in \cite{Rans} or \cite[(3.7) page
46]{SaffTotik}. Historically, it was first Fekete who proved
the inequality. Before that, Faber \cite{Faber} has already proved $\max_{z\in M}
\left| \prod_{j=1}^k (z-w_j) \right| \ge \rho(M)^k$ for $M$ a
Jordan domain bounded by an analytic Jordan curve and with $\rho(M)$ standing for the conformal radius of the domain $M$.
Note that for such domains the conformal radius is equal to the transfinite diameter.
Following Fekete, Szeg\H o showed that the condition of
$\CC\setminus M$ being simply connected is not necessary, and
that with the so-called Robin constant $\gamma(M)$ (equivalent
to capacity), the stated inequality holds true, moreover,
$\gamma(M)=\Delta(M)$ in general for all compacta. For a more detailed discussion of the development of the topic see \cite{GR-2}.
\end{proof}

We will use disconnected, linearly Jordan mesurable sets in our arguments. For their transfinite diameter, a classical estimate is well-known.

\begin{lemma}{\bf (P\'{o}lya, see~\cite[Ch. VII]{Goluzin}).}\label{l:transfinitediammeasure}
Let $J \subset \RR$ be any compact set, $|J|^*$ be its outer Jordan measure. Then
$ \Delta(J)\ge |J|^*/4$.\end{lemma}


Next, we record a Nikolskii-type estimate, which is similar to the well-known analogous inequality on the real line, see e.g. \cite[4.9.6 (36)]{Timan}. For the exact version below, see Lemma 1 of \cite{GR-2} .

\begin{lemma}\label{l:Nikolskii} For any $q>0$ and any polynomial of degree at most $n$ we have that
\begin{equation*}
\|p\|_{L^q(\DK)} \ge  \left(\frac{d}{2(q+1)}\right)^{1/q} ~\|p\|_{L^\infty(\DK)} ~ n^{-2/q} .
\end{equation*}
\end{lemma}

Recall the above definition \eqref{eq:Hsetdef}. It is important for us  because we can restrict ourselves to the points of $\HH$
and neglect whatever happens for points belonging to its complement $\Gamma \setminus \HH$. Indeed, $\Gamma$ is contained in a disk of radius $d$ around any point of $K$, whence by the well-known property\footnote{A reference is \cite[p. 52, Property 5]{BF} about surface area,
presented as a consequence of the Cauchy formula for surface area.} of convex curves,
$|\Gamma|\le 2\pi d$, and the above Lemma \ref{l:Nikolskii} furnishes
$$
\int_{\Gamma\setminus \HH} |p|^q \le
  \frac{ 2\pi d \lambda^q}{ n^{2} } \|p\|^q_\infty \le 4\pi(q+1) \lambda^q  \|p\|_q^q =
  \frac{1}{2} \|p\|_q^q ,
$$
so that we find
$$
\int_{\HH} |p|^q  = \int_{\Gamma} |p|^q   - \int_{\Gamma\setminus \HH} |p|^q  \ge \|p\|_q^q - \frac{1}{2} \|p\|_q^q \ge \frac{1}{2} \|p\|_q^q.
$$
Therefore we can restrict to (lower) estimations of $|p'(\ze)|$ on the set $\HH$ where $p$ is assumed to be relatively large (compared to its maximum norm), so that we can assume that
$$
\begin{aligned}
\log\dfrac{\|p\|_\infty}{|p(\ze)|}& \le \log ( \lambda^{-1} n^{2/q} ) =
\frac{\log (1+q)}{q} + \frac{\log\left(8\pi\right)}{q}
+ \frac2{q} \log n \\
&\le \log(16\pi) + 2 \log n < 4\log n
\end{aligned}
$$
for all $q\ge 1$ and $n \ge \sqrt{16\pi}$, i.e. already for $n\ge 8$. In case $0<q<1$, a similar estimate with a factor $1/q$ obtains even more easily. Summing up we have

\begin{lemma}
Let $n \in \NN$ and $0<q<\infty$. Let $p \in \PP_n$ and $\HH=\HH(p) \subset \DK$ be defined according to \eqref{eq:Hsetdef}. Then  we have
\begin{equation*}
\int_{\HH} |p|^q \geq \frac12 \|p\|^q_{L^q(\DK)}.
\end{equation*}
Furthermore, for any point $\ze \in \HH$ we also have
\begin{equation*}
 \log \frac{\|p\|_\infty}{|p(\ze)|} \le \frac4{\min(1,q)} \log n ~~ ( n\ge 8).
\end{equation*}
\end{lemma}
Note that for $q\ge 1$ this was given exactly as here already in Lemma 2 of Section 5.3 of \cite{GR-2}, so here we repeated the easy deduction for the reader's convenience only.

\section{Polygonal convex domains}

In the following we denote $D_r(z_0):=\{ z \in \CC~:~ |z-z_0|\le r\}$. As a key step towards the proof of Theorem \ref{th:polygon}, we want to prove in this section the following partial result.
\begin{theorem}\label{th:acute-alpha}
Let $K\Subset \CC$ be a convex polygonal domain, $0<q<\infty$, and $n\in \NN$.
Let $U$, $V$, $W$ be consecutive (in the counter-clockwise direction) vertices of $K$,
and assume that $\alpha:= \angle(UVW)$ is an acute angle of $K$ at $V$, i.e.  $0<\alpha<\pi/2$.

Then there exists a positive constant $\mu(\alpha):=\mu(\alpha,q)>0$ with the following property.
For any $p \in \Pn(K)$ and for any $0 < r \le R_V/8$, where we put
\begin{equation*}
R_V:= \dfrac{1}{64}\min(|VU|, |VW|),
\end{equation*}
it holds
\begin{equation}\label{acutepprimep}
\mu(\alpha) ~n^q  \int_{\HH\cap D_{r}(V)} |p|^q  \le
\int_{\Gamma\cap D_{8r}(V)} |p'|^q .
\end{equation}
\end{theorem}

\subsection{On one side of the acute angle, $|p'(z)|\gg n |p(z)|$ holds true}

\begin{lemma}\label{l:oneside}
With the notations of Theorem \ref{th:acute-alpha}
let us define the segments
\begin{equation}\label{newIplusminusF}
I_+=[V,U]\cap D_{R_V}(V), \quad I_{-}=[V,W]\cap D_{R_V}(V).
\end{equation}
Then we have for any $p\in \PK$ that either $(I_{+}\cup I_{-}) \subset \Gamma \setminus \HH $, or at least for one of the two segments $I_{+}$ or $I_{-}$ we have for all points $z\in I_{\pm}$ of the relevant segment the inequality
\begin{equation}\label{onesidecn}
|p'(z)| \ge \frac{\sin (\alpha/2)}{4 d} n |p(z)|
\ge \frac{\sin \alpha}{8d} n |p(z)| \qquad (z \in I_{\pm}).
\end{equation}
\end{lemma}

Note that Lemma 5 of \cite{GR-3} is a somewhat  similar statement  in a different context.

\begin{proof}
Without loss of generality we may assume that $V=0$ and that the segments $I_+$ and $I_-$ are in conjugate directions, that is, with $\beta:=\alpha/2$, we have
\begin{equation}\label{newIplusminus}
I_{\pm}:=\left[0, \,    R_V e^{\pm\beta i}\right].
\end{equation}
Let us consider the homothety about $V=0$ and of ratio $1/64$. It brings $K$ to $K'=\dfrac1{64}K$,
whose  boundary includes $I_\pm$.
By convexity, $K'\subset K$, and for the diameters $d':=d(K')=\frac{1}{64}d(K)$.

Suppose  first that some $p\in \PK$ has  $m\ge n/2$ zeroes in $K'$. Without loss of generality we may assume that $p$ is monic.
We have for any point
$z\in K'$ the estimate $|p(z)|  \le (d')^{m}d^{n-m}\le (d')^{n/2}d^{n/2}.$
To estimate from the other side, we take  any segment $I \subset K$ of the length $d$
and apply  Lemma~\ref{l:transfinitediam}, and  Lemma~\ref{l:transfinitediammeasure}:
$\|p\|_K \ge \|p\|_I \ge \Delta(I)^n \ge (d/4)^n.$
After dividing these two estimates we get
$$
\dfrac{|p(z)|}{\|p\|_{K}}\le \frac{(d')^{n/2}d^{n/2}}{(d/4)^n} = \left(\dfrac{d'}{d}\right)^{n/2}4^n= \left(\dfrac{1}{64}\right)^{n/2}4^n= 2^{-n}.
$$
for all $z\in K'$, in particular also on both segments \eqref{newIplusminus}. Whence for $n\ge n_0$ ($n_0$ is necessary for $2^{-n} <\lambda n^{-2/q}$ with $\lambda<1$ holds only for $n\ge n_0$) we must have $I_{\pm}\cap \HH =\emptyset$.

Denote $\HP:=\{z\in \CC~:~ \Re z\ge 0\}$ the upper halfplane. Let now assume that $\HH\cap (I_{+}\cup I_{-} )\ne \emptyset$. It follows that there are $m<n/2$ zeroes of $p$ in $K'$, whence there are either $\ge n/4$ zeroes in $\HP\cap (K\setminus K')$, or there are $\ge n/4$ zeroes in $-\HP\cap (K\setminus K')$.

Write $\Z=\{z_1,\ldots,z_n\}$ for the set of zeroes of $p$, all in $K$ by assumption. By symmetry, it suffices to deal with one case, say when $\W:=\Z \cap (-\HP\cap (K\setminus K'))$ contains $\ge n/4$ zeroes. Fix any point $z\in I_{+}$, with tangent direction $e^{i\beta}$, and write for all $z_j\in \Z$ $z_j=z+r_je^{i\varphi_j}$ ($j=1,\ldots,n$).
Then the usual Tur\'an type estimate yields\footnote{Note that, here again, for any $z_j\in \Z \subset K$ the angle $\varphi_j\in [-\pi+\beta,\beta]$, whence $\beta-\varphi_j \in [0,\pi]$ and $\sin (\beta-\varphi_j) \ge 0$, therefore any term can be estimated from below by $0$ and hence can be dropped from the full sum.}
\begin{equation*}
\left| \frac{p'}{p}(z) \right| \ge \Im \left(-{e^{i\beta}\frac{p'}{p}(z)}\right) =
\sum_{z_j\in\Z} \Im \frac{e^{i\beta}}{z_j-z} =\sum_{z_j\in\Z} \frac{\sin (\beta-\varphi_j)}{r_j} \ge \sum_{z_j\in\W} \frac{\sin (\beta-\varphi_j)}{r_j}.
\end{equation*}
For one thing, it is clear that $r_j\le \diam(K)=d$. Also, by construction, clearly $-\pi/2 < \varphi_j <0$ for any $z_j\in \W$, whence for such zeroes
$\beta -\varphi_j\in (\beta, \beta+\pi/2)
\subset (\beta, \pi-\beta)$ and $\sin(\beta-\varphi_j)\ge \sin \beta$. These considerations thus yield for any $z\in I_{+}$
$$
\left| \frac{p'}{p}(z) \right| \ge \# \W \frac{\sin \beta}{d} \ge \frac{\sin \beta}{4d} n = \frac{\sin (\alpha/2)}{4d} n \ge \frac{\sin \alpha}{8d} n .
$$

\end{proof}

\subsection{A crucial step}
Here we prove the following key lemma.

\begin{lemma}[{\bf Tilted normal estimate}]\label{l:TNE}
Let $K\Subset \CC$ be a convex polygonal domain,
$U,$ $V,$ $W>0$ be consecutive vertexes of $K$ (in the
counter-clockwise direction), $$0<\alpha:=\angle(U,V,W) <\pi/2$$ be the acute angle at $V,$ and put
\begin{equation}\label{theta}
\theta:=\dfrac12 \arcsin \left(\dfrac{1}{8} \sin \alpha\right)
\quad  (which \ means \ that  \ 8 \sin (2\theta) = \sin \alpha).
\end{equation}
Suppose $\zeta\in (V,W)$, and that $|V\ze| \le |VU|/8$. Let $D=D(\zeta)$ be
the point of intersection\footnote{The intersection point exists \emph{inside the interval $(V,U)$}, for according to the sine theorem $f(\alpha):=|VD|:|V\ze|=\sin(\pi/2-2\theta) : \sin(\pi/2-\alpha+2\theta)=\cos (2\theta) : \cos(\alpha-2\theta)$, and therefore $f(\alpha) \le 1/\cos(\alpha-2\theta)$, which is strictly increasing with $\alpha$ and $2\theta:=\arcsin (\sin(\alpha)/8)$, and reaches its maximum at $\alpha=\pi/2$ with a value $1/\cos(\pi/2-\arcsin(\sin(\pi/2)/8))=1/\sin(\arcsin(\sin(\pi/2)/8))=8$. That is why we assumed $r < R_V/8$ in Theorem \ref{th:acute-alpha}.} of $(V,U]$ and the ray
$$
\zeta+ t e^{i (\pi/2+\arg(VW)+2\theta)}, \quad t\ge 0,
$$
and $T=T(\zeta)$ be the point of intersection of $(V,U]$ and the ray
$$\zeta+ t e^{i (\pi/2+\arg(VW)+3\theta)}, \quad t\ge 0.$$
Let $\omega \in (0,8/e)$ be a positive parameter.

Then for any compact, linearly Jordan measurable set $J\subset [T,D]$, with its measure
\begin{equation}\label{J}
|J| \ge \omega |D\zeta|,
\end{equation}
and
$p\in {\mathcal P}_n(K)$ we have
$$
\left| \frac{p'}{p}(\ze)\right|  \ge \dfrac{\sin \theta}{7.5 \log (8/\omega) }
\left( n \frac{\sin\theta}{d} -\frac{2}{|T\zeta|} \log \frac{\|p\|_J}{|p(\zeta)|} \right).
$$

Consequently, if $|p(\zeta)|\ge \|p\|_J$, then
\begin{equation}\label{Tiltedestimate}
\left| \frac{p'}{p}(\ze)\right|  \ge
\dfrac{ \sin^2 \theta}{7.5 d \log (8/\omega)} \,n
> \dfrac{\sin^2 \alpha}{2000 d \log (8/\omega)}\, n =: C(\alpha,\omega)\, n.
\end{equation}
\end{lemma}

\begin{proof}
{\bf 1.} Without loss of generality we suppose that $[V, W] \subset \RR,$  and $\zeta=0$.
Applying the sine theorem, we obtain
\begin{equation*}
b:=|D\zeta|=\frac{\sin \alpha}{\sin(\pi/2-\alpha+2\theta)}|V\zeta|=
\frac{\sin \alpha}{\cos(\alpha-2\theta)} |V\zeta|\le  \frac{\sin \alpha}{\sin(2\theta)} |V\zeta| =8 |V\zeta|,
\end{equation*}
\begin{equation}\label{T}
a:=|T\zeta|=\frac{\sin \alpha}{\sin(\pi/2-\alpha+3\theta)}|V\zeta|=
\frac{\sin \alpha}{\cos(\alpha-3\theta)}|V\zeta| \ge  |V\zeta|\sin \alpha.
\end{equation}
Clearly, $a<b$ and for all $\xi\in[T,D]$ we have $a \le |\xi|\le b.$

Recall that we have denoted the zeroes of $p \in \PK$ as  $\Z=\Z(p)$ and we have written
$\Z=\{z_j=r_je^{i\ff_j}, \ j=1,\ldots,n\}.$
Further, let us denote
$$
S(\ff,\psi):=\{z\in \CC~:~ \ff < \arg z < \psi\}\qquad \text{\rm and} \qquad  \Z(\ff,\psi):=\Z \cap S(\ff,\psi),
$$
and similarly for closed, or half-open, half-closed etc. intervals for the arguments.
We split the set $\Z$ into the following parts.
\begin{align*}
\Z_1:&= \Z[0,\theta]\,,
&& n_1:=\#\Z_1 
, \notag \\
\Z_{2}:&= \Z(\theta,\pi-\theta) \cap D_{2b}(0)\,,
&&n_2:=\#\Z_{2}, \notag \\
\Z_{3}:&= \Z(\theta,\pi-\theta) \setminus D_{2b}(0)\,,
 && n_3:=\#\Z_{3}, \notag \\
\Z_4:&= \Z[\pi-\theta,\pi]\,,
&& n_4:=\#\Z_4 
\notag ~.
  \end{align*}
The partitioning of the set $\Z$ and, as a consequence, of the polygonal domain  $K$ is depicted in
Fig.~\ref{fig:1}.
\begin{figure}[h]
    \centering
   \includegraphics[width=420pt, keepaspectratio]{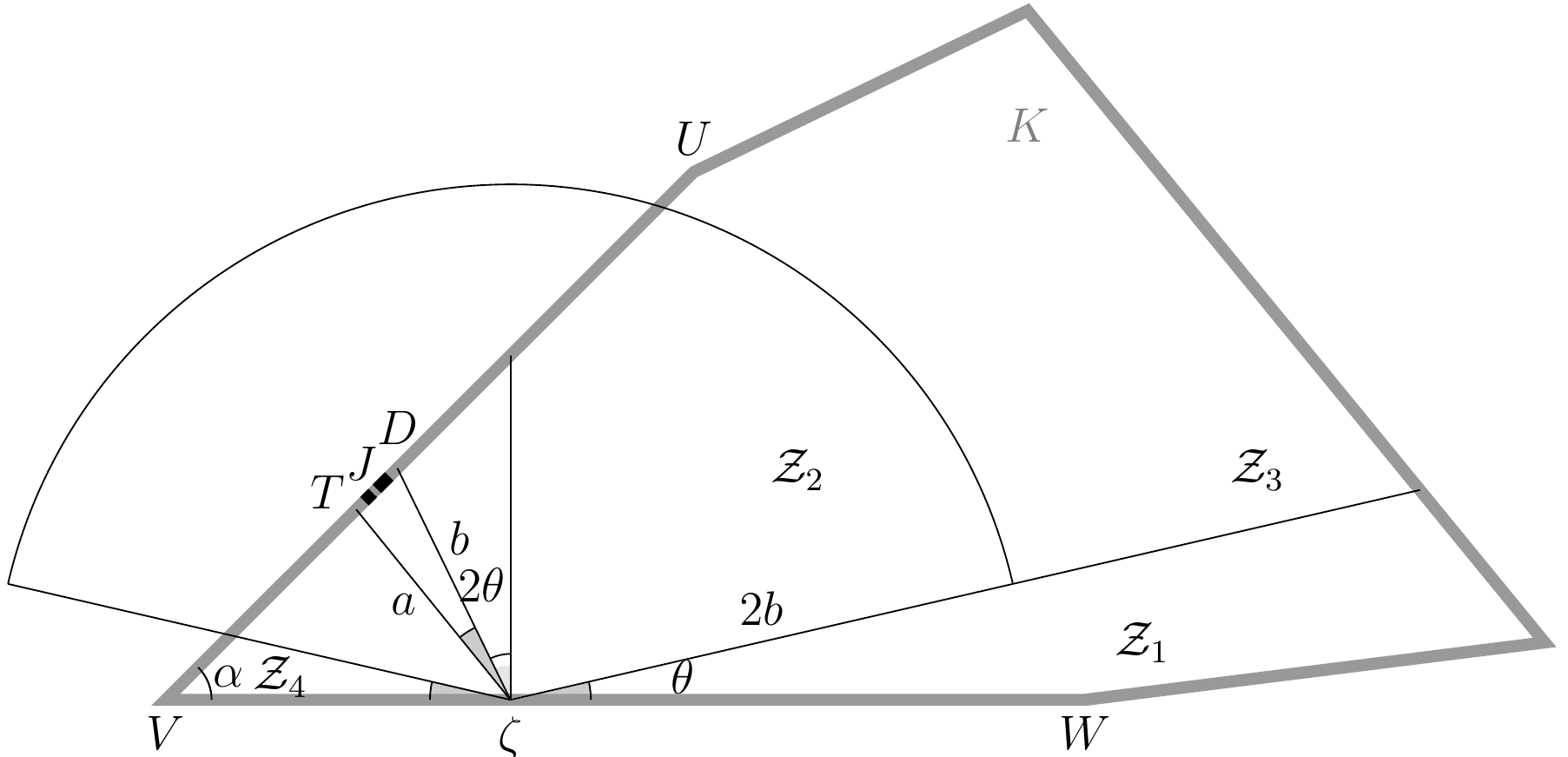}
    \caption{Partition of the set $\Z$}
    \label{fig:1}
\end{figure}

In the following, we estimate $\left|\dfrac{p(\xi)}{p(0)}\right|$ for any $\xi\in J$ from below.
It will be convenient to use the notation
$$\xi=\rho e^{i(\pi/2+\phi)}, \quad 2\theta \le \phi \le 3\theta, \quad  a\le \rho\le b.$$

{\bf 2.} Let us start with the estimate of the distance of any $z=re^{i\varphi} \in \Z_1$ from $J$.
By the cosine theorem, we obtain  the estimate
\begin{equation*}
\begin{aligned}
\left|\dfrac{z-\xi}{z}\right|^2 &=
\dfrac{r^2+\rho^2-2r\rho\cos(\pi/2+\phi-\varphi)}{r^2} \\& \ge 1+\dfrac{2\rho}{r} \sin(\phi-\varphi)
\ge 1+\dfrac{2\rho}{r} \sin \theta
\ge 1+\dfrac{2 a}{d} \sin \theta,
\end{aligned}
\end{equation*}
in view of $r\le d$.
As $\dfrac{2a}{d}\sin \theta \le \dfrac{2a \sin \alpha}{8d}\le \dfrac{a}{4d}\le \dfrac{1}{2}$,  we can use
\begin{equation*}
1+x >  e^{x/2} \quad \mbox{for} \quad 0<x\le 1/2
\end{equation*}
to get
$$
\left|\dfrac{z-\xi}{z}\right|^2 \ge \exp \left(\dfrac{a\sin \theta}{d}\right).
$$
Applying this estimate for all the $n_1$ zeroes $z_j\in \Z_1$ , we finally come to the inequality
\begin{equation}\label{Z4contri}
\prod_{z_j\in\Z_1}\left|\frac{z_j-\xi}{z_j}\right| \ge
 \exp
\left(n_1 \dfrac{a\sin \theta}{2d}\right) ~\qquad
\left(\xi  \in J \right).
\end{equation}

{\bf 3.}
Now we estimate the distance of any $z=re^{i\varphi} \in \Z_4$ from $J$. Let $X$ be the intersection of the side $UV$ and the ray emanating from $\zeta$ with an angle $\pi-\theta$ from $\overrightarrow{VW}$. Then in the triangle $\triangle(V\zeta X)$ the sides compare as
$$
\dfrac{|\zeta X|}{|\zeta V|}=\dfrac{\sin \alpha}{\sin(\alpha+\theta)}= \dfrac{\sin \alpha}{\sin \alpha\cos\theta+\cos\alpha\sin\theta}\le \dfrac{1}{\cos\theta},
$$
hence by convexity any point is at most as far from $\zeta$ as $\max(|V\zeta|,|\zeta X|) \le |V\zeta| /\cos \theta$. As $\Z_4 \subset \triangle(V \zeta X)$, we find $r \le |V\zeta|/\cos\theta$.
By the cosine theorem again, and then using $\cos(\varphi-\pi/2- \phi) \le  \sin(4\theta)$ we get
$$
\left|\dfrac{z-\xi}{z}\right|^2=  \dfrac{r^2+\rho^2-2 \rho r\cos(\varphi-\pi/2- \phi)}{r^2}\ge  1+ \dfrac{\rho}{r}\left(\dfrac{\rho}{r}-2\sin(4\theta)\right).
$$
To estimate the last quantity we use the following observations.
First, $\dfrac{\rho}{r}\ge \dfrac{a }{d }$. Second, by $r \le |V\zeta|/\cos \theta$ and \eqref{T} we find $\dfrac{\rho}{r}\ge \dfrac{a }{|V\zeta|/\cos \theta }\ge \sin \alpha \cos \theta$.
Third, \eqref{theta} implies
\begin{align*}
\sin \alpha \cos \theta-2\sin(4\theta) & = 8\sin(2\theta) \cos\theta -4\sin(2\theta)\cos(2\theta)
\ge  4\sin(2\theta) \cos(2\theta).
\end{align*}
Hence, $\cos \theta \cos(2\theta) \ge \cos^2(2\theta) = \frac12 (1+\cos(4\theta)) \ge 1/2$ yields
\begin{align*}
\left|\dfrac{z-\xi}{z}\right|^2  & \ge 1+ \dfrac{a}{d } \left(\sin \alpha \cos\theta-2\sin(4\theta)\right)
\\
&\ge 1+ \dfrac{a 8 \sin \theta \cos\theta \cos(2\theta)}{d}\ge
1+ \dfrac{4a\sin\theta }{d}.
\end{align*}
The same reasoning  as in the previous case gives \begin{equation}\label{Z1zeros}
\prod_{z_j\in\Z_4}\left|\frac{z_j-\xi}{z_j}\right| \ge
\exp \left(n_4\dfrac{a\sin\theta }{d}\right)\ge
\exp \left(n_4\dfrac{a\sin\theta }{2d}\right)~\qquad
\left(\xi \in J \right).
\end{equation}

{\bf 4.}  Next we estimate the contribution of zero factors belonging to $\Z_{3}$,
i.e. the ``far" zeroes
$z=re^{i\varphi}$ for which we have $r\ge 2 b $  and $\varphi\in (\theta,\pi-\theta).$
Applying the estimate $1-x\ge e^{-2x}$ ($0<x<1/2$)  we see that
   \begin{equation*}
\begin{aligned}
\left|\dfrac{z-\xi}{z}\right|^2 & =
\left|1-\dfrac{\xi}{z}\right|^2\ge \left(1-\dfrac{\rho}{r}
\right)^2
\ge \left(1-\dfrac{b}{r}\right)^2 \ge \exp\left(-4\dfrac{b}{r}\right).
\end{aligned}
\end{equation*}
This yields
\begin{equation}\label{Z3plusfin}
\prod_{z_j\in\Z_{3}}\left|\frac{z_j-\xi}{z_j}\right| \ge
\exp \left(-2b\sum_{z_j\in \Z_3} \dfrac{1}{r_j} \right)\ge
 \exp\left( - \dfrac{2b}{\sin \theta}\sum_{z_j\in \Z_3}\dfrac{\sin\varphi_j}{r_j}\right) \qquad (\xi \in J ).
 \end{equation}

{\bf 5.} Finally we consider the contribution of the zeroes from $\Z_{2}$.
By $ \Delta(J) \ge  |J|/4$ (in view of  Lemma \ref{l:transfinitediam}, and  Lemma~\ref{l:transfinitediammeasure}),  $\dfrac{2b}{r_j}\dfrac{\sin \varphi_j}{\sin \theta }
\ge 1$ for $z_j=r_je^{i\varphi_j}\in \Z_2$,  and \eqref{J},
we find
\begin{equation}\label{Z2}
\begin{aligned}
\max_{\xi\in J} \prod_{z_j\in\Z_{2}}\left|\frac{z_j-\xi}{z_j}\right| &\ge
\left(\Delta(J)\right)^{n_2} \prod_{z_j\in\Z_{2}}
\frac{1}{2b} \ge \left(\frac{|J|}{8 b}\right)^{n_2}=
\exp\left(-n_2 \log (8/\omega) \right)\\
&\ge \exp\left(- \log (8/\omega)\dfrac{2b}{\sin \theta} \sum_{z_j\in \Z_2}\dfrac{\sin\varphi_j}{r_j}\right).
\end{aligned}
\end{equation}

{\bf 6.}
If we collect the estimates \eqref{Z4contri}, \eqref{Z1zeros},
 \eqref{Z3plusfin} and \eqref{Z2} we obtain
for a certain point of maxima $\xi_0\in J$ in \eqref{Z2} the
inequality
\begin{align*}
\frac{\|p\|_J}{|p(0)|} & = \prod_{z_j\in\Z} \left|\frac{z_j-\xi_0}{z_j}\right|
\\
& \ge \exp\left\{ (n_1+n_4) \dfrac{a\sin\theta}{2d}
- \dfrac{2b \log (8/\omega) }{\sin \theta}
\sum_{z_j\in \Z_2}\dfrac{\sin\varphi_j}{r_j}-
\dfrac{2b}{\sin \theta} \sum_{z_j\in \Z_3}\dfrac{\sin\varphi_j}{r_j} \right\}
\\
& \ge \exp\left\{ n \dfrac{a\sin\theta}{2d}
- \left(\dfrac{2b\max(1, \log (8/\omega))}{\sin \theta}+\dfrac{a}{2}\right)
\sum_{z_j\in \Z_2 \cup \Z_3}\dfrac{\sin\varphi_j}{r_j}\right\}
\\
& \ge \exp\left\{ n \dfrac{a\sin\theta}{2d}
- \left(\dfrac{2b\max(1, \log (8/\omega))}{\sin \theta}+\dfrac{a}{2}\right)
\sum_{z_j\in \Z}\dfrac{\sin\varphi_j}{r_j}\right\},
\end{align*}
taking into account that $n_1+n_2+n_3+n_4=n$ and that for each zeroes in $\Z_2\cup\Z_3$ we have $\dfrac{\sin \varphi_j}{r_j} \ge \dfrac{\sin \theta}{d}$.

After taking logarithms, using $a\le b$ and the assumption $\omega \le 8/e$ we arrive at
$$
\log \frac{\|p\|_J}{|p(0)|} \ge  n \frac{a\sin\theta}{2d} - \dfrac{2.5 b \log (8/\omega) }{\sin \theta}\left| \frac{p'}{p}(0)\right| ~,
$$
that is, writing in again the normalization $\zeta:=0$,
\begin{equation*}
\begin{aligned}
\left| \frac{p'}{p}(\ze)\right| & > \dfrac{a\sin \theta}{5 b \log (8/\omega) }\left( n \frac{\sin\theta}{d} -\frac{2}{a} \log \frac{\|p\|_J}{|p(0)|} \right).
\end{aligned}
\end{equation*}

It remains to prove that $a/b\ge 2/3$, this clearly furnishing the statement.
Indeed, let
$$f(\alpha):=\dfrac{a}{b}=\dfrac{\cos(\alpha-2\theta(\alpha))}{\cos(\alpha-3\theta(\alpha))}.
$$
Straightforward calculations give
$$\begin{aligned}
\cos^2(\alpha-3\theta)f'(\alpha)=&-\sin(\alpha-2\theta)\cos(\alpha-3\theta)(1-2\theta')\\ \phantom{=}&+\cos(\alpha-2\theta)\sin(\alpha-3\theta)(1-3\theta')
\\=&
\sin(-\theta)(1-2\theta')-\cos(\alpha-2\theta)\sin(\alpha-3\theta)\theta'<0.
\end{aligned}
$$
This clearly entails, in view of $\sin ((3/2)x)\le (3/2)\sin x$, $x\in[0, \pi/2]$, that it holds
$$
\dfrac{a}{b}\ge f(\pi/2)=
\dfrac{\cos(\pi/2-2\theta(\pi/2))}{\cos(\pi/2-3\theta(\pi/2))}=
\dfrac{\sin(\arcsin(1/8))}{\sin((3/2)\arcsin(1/8))} \ge  \dfrac23.
$$
\end{proof}

\subsection{Proof of Theorem~\ref{th:acute-alpha}}
Now we use a different positioning of our $K$ in the complex plane, namely without loss of generality we suppose that
$[V, W] \subset \RR$ and $V=0$, $W>0.$

{Take any $p\in \PK.$ First, if $(I_{+}\cup I_{-}) \subset \Gamma \setminus \HH $, then the left hand side of \eqref{acutepprimep} vanishes and there remains nothing to prove.
So consider the case when $(I_{+}\cup I_{-})$ has points from $\HH$. In view of Lemma~\ref{l:oneside} we may assume that
inequality~\eqref{onesidecn} is valid on $[0, e^{i\alpha} R_V]$ (which corresponds to $I_{+}$ from  \eqref{newIplusminusF}), containing $J:=[0, 8re^{i\alpha}]$.} That is, we have
\begin{equation}\label{onesidecn_s}
|p'(z)| \ge \frac{\sin \alpha}{8d} n |p(z)| \qquad (z \in J:=[0, 8 r e^{i\alpha}].
\end{equation}
Hence in particular we also have
\begin{equation}\label{oneside_int_cn_s}
\int_{J} |p'|^q  \ge \left(\frac{\sin \alpha}{8d}\right)^q n^q \int_{J} |p|^q .
\end{equation}

Now we see to the estimation of $|p'(s)|$ on $I:=[0,r] \subset [0,R_V/8]$, which  corresponds to $I_{-}$ from \eqref{newIplusminusF}.
First of all we introduce a number $\eta$ such that
\begin{equation}\label{etaint}
\int_{[0,\eta]}|p|^q=\dfrac 12 \int_{I}|p|^q,
\end{equation}
the parameter
$\omega_\alpha:=2^{-5} \sin \alpha$,
and the set
\begin{equation}\label{Sdef}
\begin{gathered}
 S=\{x\in [\eta,r] \ \colon \ |p'(x)|\le \kappa n|p(x)| \}, \qquad\textrm{with}
\\ \kappa:=\kappa(\alpha):=C(\alpha,\omega_\alpha)=\dfrac{\sin^2 \alpha}{2000 \,d\log (8/\omega_\alpha)  }=\dfrac{\sin^2 \alpha}{2000 \,d \log (2^8/\sin \alpha)  },
\end{gathered}
\end{equation}
where $C(\alpha, \omega)$ is the quantity from \eqref{Tiltedestimate}.
By definition of $S$ we immediately have
\begin{equation}\label{complSint}
  \int_{[\eta,r]\setminus S}|p'|^q\ge \kappa^q n^q \int_{[\eta,r]\setminus S}|p|^q .
\end{equation}

 It remains to majorize  $\Ds\int_S|p|^q.$
 This will be pursued not by seeking a comparison with $\Ds\int_{S} |p'|^q$, but with comparing to $ \Ds \int_{J} |p|^q$. When succeeding, we will be able to use that on the whole $J$ we already have the strong majorization \eqref{onesidecn_s} and \eqref{oneside_int_cn_s} of $|p(z)|$ by $|p'(z)|$.

 As $x$ is real, and $|p(x)|^2$ and $|p'(x)|^2$  are polynomials, the set $S$ consists of a finite number of segments, some of which may be degenerate.
 We dissect each segment of $S$ to essentially disjoint segments with  lengths not exceding $\eta$. Thus  we get a representation
 $$
S=\bigcup_{\ell=1}^{m} I_\ell,
\qquad  I_\ell:=[u_{\ell}, v_\ell], \qquad |I_\ell|=v_\ell-u_{\ell} \le \eta, \quad \ell=1,\ldots,m,
$$
$$u_1\le v_1 \le u_2\le  \ldots \le u_m\le v_m.
$$
For a point $x\in S$ we introduce the notations, corresponding to that of Lemma \ref{l:TNE}:
$$
D(x):=J \cap \{x+te^{i(\pi/2+2\theta)} t~:~t>0\}, \qquad
T(x):= J \cap \{ x+te^{i(\pi/2+3\theta)} t~:~t>0\},
$$
$$
b(x):=|D(x)-x|= \frac{\sin \alpha}{\sin(\pi/2-\alpha+2\theta)} x = \frac{\sin \alpha}{\cos(\alpha-2\theta)} x,
$$
$$
h(x):= |T(x)-D(x)|=
\frac{\sin \theta}{\cos(\alpha-3\theta)} b(x) =\frac{\sin \alpha \sin \theta}{\cos(\alpha-2\theta)\cos(\alpha-3\theta)} x \ge  2^{-4} (\sin\alpha)^2x .
$$
(Recall that the given rays emanating from $x \in S$ intersect $J$, not only its straight line, as was said in Footnote 8 when formulating Lemma \ref{l:TNE}.)

 Let $s_\ell=\arg \max \{|p(x)| \colon x\in I_\ell \}.$
 We define the points $D_\ell:=D(s_\ell)$, $T_\ell:=T(s_\ell)$
 and also introduce $b_\ell:=b(s_\ell)$, $h_\ell:=h(s_\ell)$,
and two subsets of the segment $[T_\ell,D_\ell]\subset J$
 $$J_\ell:=\{z\in [T_\ell,D_\ell] \colon |p(z)|<|p(s_\ell)|\}, \quad J_\ell^c=[T_\ell,D_\ell]\setminus J_\ell.$$
 We claim that $|J_\ell|< h_\ell/2$ (and hence $|J_\ell^c| \ge  h_\ell/2$).
 Indeed, if $|J_\ell| \ge  h_\ell/2$ then we would have by an application of Lemma \ref{l:TNE} with the set $J_\ell$ and the parameter
 $$\omega:=\dfrac{|J_\ell|}{b_\ell}\ge \dfrac{h_\ell}{2b_\ell}=
 \frac{\sin \theta}{2\cos(\alpha-3\theta)} \ge \dfrac{\sin \theta}{2} > \dfrac{\sin 2\theta}{4}  = 2^{-5}\sin \alpha =\omega_\alpha $$ the inequality
 $\kappa n|p(s_\ell)| = C(\alpha,\omega_\alpha) n|p(s_\ell)| <C(\alpha,\omega) n|p(s_\ell)| <  |p'(s_\ell)|$,
 which contradicts to \eqref{Sdef}.

Our goal is to construct  sets $Q_{\ell}\subset  J_\ell^c$  of linear Jordan measure $|Q_\ell| =  2^{-6}(\sin \alpha)^2 |I_\ell|$.  A~further criterion for $Q_\ell$ is that it must be disjoint from all previous $Q_i$s. This can be ascertained because
$$
\left|\bigcup_{i=1}^{\ell-1} Q_i\right| =\sum_{i=1}^{\ell-1} |Q_i| =
2^{-6}(\sin \alpha)^2 \sum_{i=1}^{\ell-1} |I_i| \le 2^{-6}(\sin \alpha)^2 v_{\ell-1} \le
2^{-6}(\sin \alpha)^2 s_\ell \le h_\ell/4,
$$ while
\begin{align*}
\left|J_\ell^c\setminus  \bigcup_{i=1}^{\ell-1} Q_i\right| &\ge
|J_\ell^c|-\sum_{i=1}^{\ell-1} |Q_i| \ge
 h_\ell/4 \ge 2^{-6}(\sin \alpha)^2 s_\ell \ge 2^{-6}(\sin \alpha)^2 \eta \ge 2^{-6}(\sin \alpha)^2 |I_\ell|.
\end{align*}
Therefore, --taking into account that for all points $z\in Q_\ell \subset J_\ell^c$ we have by construction $|p(s_\ell)| \le |p(z)|$--  we are led to
$$
\int_{I_\ell} |p|^q   \le |p(s_\ell)||I_\ell| \le
2^6(\sin\alpha)^{-2} \int_{Q_\ell} |p|^q , \quad \ell=1,\ldots,m,
$$
and
\begin{equation}\label{Sint}
\int_S |p|^q\le 2^6(\sin\alpha)^{-2}\int_J |p|^q
\end{equation}

Finally, collecting  \eqref{onesidecn_s},  \eqref{etaint}, \eqref{complSint}, and \eqref{Sint},  we are led to
$$
\begin{aligned}
\int_{I\cup J}|p|^q&=2\int_{[\eta,r]\setminus S}|p|^q+2\int_{S}|p|^q+\int_{J}|p|^q
\le  2\int_{[\eta,r]\setminus S}|p|^q+(1+2^7(\sin\alpha)^{-2})\int_{J}|p|^q
\\&\le \left\{2 \kappa^{-q} + (1+2^7(\sin\alpha)^{-2}) \left( \frac{8d}{\sin \alpha }\right)^q\right\} n^{-q}  \int_{I\cup J}|p'|^q.
\end{aligned}
$$

\section{Proof of Theorem~\ref{th:polygon}}

Having Theorem \ref{th:acute-alpha} proven, finally in this section we derive the main result of the paper.

Observe that if $K$ is a convex polygon, and
$$
\delta_0>0
$$
is chosen to be a small enough constant (depending on $K$, of course), then we can fix some other (small) positive constant parameter
$$
h_0>0
$$
such that a point $\ze \in \DK$ is either $\delta_0$-close to some vertex $V$ where $K$ has an acute angle (and $\ze$ lies on one side of the polygon ending in the vertex $V$), or $\ze$ has a local depth $h(\ze,K)\ge h_0$.

This can be made explicit in dependence of a few further geometric parameters of the polygonal convex domain $K$. Indeed, consider any side $[A,B]$ of $K$, any point $Z\in [A,B]$ on this side, and denote the angles at $A$ and $B$ by $\alpha$ and $\beta$, respectively. Let us draw the normal chords $[A,A']$ and $[B,B']$ of $K$ at $A$ and $B$, which are perpendicular to $[A,B]$. There are three cases corresponding to the number of acute angles among $\alpha, \beta$ being 0, 1 or 2.

If both angles are at least $\pi/2$, then both normal chords are of positive length. Drawing the normal chord of $Z$, perpendicular to the side $[A,B]$ thus remains at least as long as the normal chord of the orthogonal trapezium $A'ABB' \subset K$. This is at least as long as the minimum of the distances $|AA'|$ and $|BB'|$.

If only one angle is an acute angle, say $\alpha$, then $|BB'|$ is still positive and the orthogonal triangle $\triangle(ABB')$ is included in $K$ by convexity. Therefore by simple similarity $h(Z,K) \ge \frac{|AZ|}{|AB|} |BB'|\ge \delta_0 \frac{|BB'|}{|AB|}$ if $|AZ|\ge \delta_0$.

Finally, assume that both angles $\alpha$ and $\beta$ are acute angles. Drawing the other side lines of $K$ with angles $\alpha$ and $\beta$ at $A$ resp. $B$ to $[A,B]$, we obtain an intersection point $C$ and a triangle $\triangle(ABC)$, which contains $K$. Given that $w:=w(K)>0$, there must exist a point $Q$ of $K\subset \triangle(ABC)$ at least of distance $w$ from the line of $[A,B]$. In view of $Q \in K \subset \triangle(ABC)$, the ray emanating from $Q$ perpendicularly to $[A,B]$ intersects $[A,B]$ in a point $M$ strictly in $(A,B)$. Note that by convexity $\triangle(ABQ)\subset K$.

Now if $Z$ is any point of $(A,B)$ -- say belonging to $(A,M]$ -- then the normal chord at $Z$ is parallel to $MQ$, and as $\triangle(ABQ)\subset K$, by similarity we find $h(Z,K)\ge \frac{|AZ|}{|AM|} |MQ| \ge |AZ| \frac{w}{|AB|} \ge \delta_0 \frac{w}{|AB|}$ if $|AZ|\ge \delta_0$.

Summing up, we find that in all cases $h(Z,K)\ge h_0$ unless $Z$ is $\delta_0$-close to some vertex with an acute angle of $K$ at this vertex.

\medskip
Now the full set $\Gamma:=\DK$ splits into two parts: the set of points $\Gamma_0$ which are $\delta_0$-close to some acute angle vertex, and the remaining set $\Gamma\setminus \Gamma_0$, with local depth of its points exceeding the preset positive value $h_0$.

Generally speaking, for the latter we can directly involve the local depth Theorem \ref{th:localdepth} to derive $|p'(\ze)| \gg n |p(\ze)|$ whenever $\ze \in \HH$ and $n\ge n_0$, with the implied constant depending only on the values of $d$ and $h_0$, while for $\Gamma_0$--so where $|\ze-V|<\de_0$ with some vertex $V$ with acute angle--we can invoke Theorem \ref{th:acute-alpha}.

Let us execute this program. To make our life easier, assume that the neighborhoods $D_V(8\delta_0)$ of acute angle vertices remain disjoint -- this is certainly achieved when $\delta_0\le \min\{R_V \ : \ V~\textrm{is an acute angle vertex of}~K \}$.

So we choose $\delta_0:=\frac18\min_V R_V$, pick the corresponding $h_0=h_0(\delta_0)>0$, and denote $c_0:=c_0(K,\delta_0):=\frac{h_0^4}{1500d^5}$.

Then we can write
\begin{equation}\label{Deeppart}
\int_{\DK} |p'|^q  \ge \int_{\Gamma\setminus \Gamma_0} |p'|^q  \ge c_0^q n^q \int_{\Gamma\setminus \Gamma_0} |p|^q.
\end{equation}

Next, let $V$ be any vertex with the corresponding acute angle $\alpha$. Then an application of
Theorem \ref{th:acute-alpha} with $r:=\delta_0$ furnishes the estimate
$$
\int_{\Gamma\cap D_{8\delta_0}(V)} |p'|^q \ge \mu(\alpha) n^q  \int_{\HH\cap D_{\delta_0}(V) } |p|^q.
$$
Summing this inequality over all vertices $V$ with acute angles, the sum of the parts on the left hand side does not exceed the full integral over $\Gamma$, because by construction for different acute angle vertices these neighborhoods are disjoint. (Otherwise, we could have referred to the fact that there are at most three acute angle vertices, yielding a constant 3 here.) On the right hand side, however, the full sum is just the integral over the totality of $\Gamma_0 \cap \HH$. So putting $\mu:=\min_K \mu(\alpha)$ (with the minimum understood over all vertices $V$ of $K$ with acute angles and $\alpha$ denoting the corresponding acute angles) yields
$$
\int_{\Gamma} |p'|^q \ge \mu n^q  \int_{\Gamma_0 \cap \HH} |p|^q.
$$
From here, taking into account \eqref{Deeppart}, we are led to
\begin{align*}
2 \int_{\DK} |p'|^q & \ge c_0^q n^q \int_{\Gamma\setminus \Gamma_0}  |p|^q + \mu n^q \int_{\HH \cap \Gamma_0} |p|^q
\\ & \ge \min(c_0^q, \mu) ~ n^q \int_{\HH} |p|^q \ge \frac12 \min(c_0^q, \mu) ~ n^q \int_{\Gamma} |p|^q .
\end{align*}
This concludes the proof of Theorem \ref{th:polygon} with the constant $c(K):=4^{-1/q} \min(\mu^{1/q}, c_0)$.

\section{Concluding remarks}

As is said above in Theorem \ref{th:sharpq}, similarly to the case of the maximum norm, also for the $L^q(\DK)$ norm any compact convex domain $K$ admits polynomials $p\in \PK$ with oscillation not exceeding $O(n)$. On the other hand we have shown for seve\-ral classes of convex domains--convex domains with fixed depth (including e.g. all smooth domains), all convex polygons, generalized Er\H od-type domains--that the order of oscillation indeed reaches $c_K n$.

A natural question--quite resembling to the question posed by Er\H od back in 1939 for the maximum norm case--is to identify those domains which indeed admit order $n$ oscillation even in $L^q(\DK)$ norm.

It has been clarified that, like in case of the maximum norm, also for $L^q$ norms the interval $\II$ behaves differently in the sense that there the order of oscillation may be as low as $\sqrt{n}$. Therefore, it is certainly necessary that \emph{some} conditions are assumed for an order $n$ oscillation. The great question is if apart from having a nonempty interior, is there need for any additional assumption? We think that probably not.

\begin{conjecture}
For all $0<q<\infty$ and for all compact convex domains $K\Subset \CC$ there exist $n_0:=n_0(q,K)$ and $c_K=C(K,q)>0$ such that for all $n\ge n_0$ and for any $p\in\PK$ we have $\|p'\|_{L^q(\DK)} \ge c_K n \|p\|_{L^q(\DK)}$.
\end{conjecture}

We are not that far from this conjecture, c.f. Theorem \ref{th:nlogn}. Nevertheless, in analysis that mere $\log n$ can be the crux of the situation. Still, encouraged by the maximum norm case and by our results for the $L^q$ norm so far, we formulated an even more precise conjecture already in \cite[Section
7, Conjecture 2]{GR-1} (and repeated in \cite{GGR}, too).

\begin{conjecture}
There exists an absolute constant $c>0$ and an $n_0=n_0(q,w,d)$ such that for all $0<q<\infty$, for all compact convex domains $K\Subset \CC$, for all $n\ge n_0$ and for any $p\in\PK$ we have $\|p'\|_{L^q(\DK)} \ge c \dfrac{w}{d^2} n \|p\|_{L^q(\DK)}$.
\end{conjecture}

\end{document}